\documentclass[11pt,a4paper,reqno]{amsart}
\usepackage{amsmath,amssymb,mathtools,fullpage,setspace,microtype,xcolor}
\usepackage{lineno}
\usepackage[colorlinks=true,citecolor=blue,urlcolor=blue]{hyperref}
\newtheorem{theorem}{Theorem}[section]
\newtheorem{lemma}[theorem]{Lemma}
\newtheorem{construction}{Construction}
\newtheorem{proposition}[theorem]{Proposition}
\newtheorem{corollary}[theorem]{Corollary}
\newtheorem{conjecture}[theorem]{Conjecture}
\theoremstyle{remark}

\newcommand{\F}{\mathcal F}
\newcommand{\B}{\mathcal B}
\newcommand{\C}{\mathcal C}

\newcommand{\FF}{\mathbb F}
\newcommand{\cF}{\mathcal F}
\newcommand{\cC}{\mathcal C}
\newcommand{\Prob}{\mathbb P}
\newcommand{\one}{\mathbf 1}
\newcommand{\arb}{\operatorname{arb}}

\DeclareMathOperator{\tr}{tr}
\title{Alon's conjecture on intersecting families of graphs}
\date{}
\author{Jie Han}
\address{School of Mathematics and Statistics, Beijing Institute of Technology, Beijing, China}
\email{{\{han.jie|bin.wang\}@bit.edu.cn}}
\author{Bin Wang}
\thanks{The first author was partially supported by the National Natural Science Foundation of China (12371341) and by the Fundamental Research Funds for the Central Universities.
The second author was supported by the China Postdoctoral Science Foundation (2026M794610).}

\begin{document}
\onehalfspacing
\begin{abstract}
Given a graph $F$, a family $\F$ of graphs on $[n]$ is \emph{$F$-intersecting} if $G\cap H$ contains a copy of $F$ for every $G,H\in\F$.
We prove that there exists an absolute constant $\varepsilon>0$ such that every $P_4$-intersecting family $\mathcal F$ satisfies $|\mathcal F|\le\left(\frac12-\varepsilon\right)2^{\binom n2}$, which resolves a conjecture of Alon.
Combined with Alon's reduction, this proves that a graph $F$ admits $F$-intersecting families of asymptotic density $1/2$ if and only if $F$ is a star forest.
Multi-colored versions are also considered.
\end{abstract}
\maketitle

\section{Introduction}

In this note all graphs are simple and labelled on $[n]$. 
The notation $P_4$ means a path on four distinct vertices and three edges, and a \emph{star forest} is a forest consisting of disjoint stars.
Throughout this note we identify a graph and its edge set.
Given a graph $F$, a family $\F$ of graphs on $[n]$ is
\emph{$F$-intersecting} if $G\cap H$ contains a copy of $F$ for every $G,H\in\F$.

An intersecting family of sets on $[n]$ contains at most $2^{n-1}$ members, since it cannot contain both a set and its complement.
This bound is attained by the family of all sets containing a fixed element.
For every $t\in \mathbb N$, Katona \cite{Katona} determined the maximum size of a $t$-intersecting family, in which we require every two members of the family intersect in at least $t$ elements.
For each fixed $t$, this maximum is asymptotically $(1/2-o_n(1))2^{n}$ \footnote{For interested readers, this bound can be achieved by e.g., the family of all sets that contain at least $\lceil(n+t)/2\rceil$ elements from $[n]$.}.
For the graph-intersecting problem, the EKR bound says that any $F$-intersecting graph family on $[n]$ for $F$ with at least one edge has at most $2^{\binom{n}{2}-1}$ members.

A particularly interesting case is when the required intersection graph is a clique.
A classical conjecture of Simonovits and S\'os from 1976 asserted that every triangle-intersecting family $\mathcal F$ of graphs on $[n]$ satisfies $|\mathcal  F| \le 2^{\binom n2-3}$, with equality attained e.g. by the family of all graphs containing a fixed triangle.
Ellis, Filmus and Friedgut \cite{EFF} proved this conjecture and showed analogous results for odd-cycle-intersecting families.
This motivates a natural conjecture that whether all $K_t$-intersecting graph families have size at most $2^{\binom n2-\binom t2}$.
More recently, Berger and Zhao \cite{BergerZhao} proved this conjecture for $t=4$, and Mani and Zhang \cite{ManiZhang} proved it for $t=5$.
Both results include stability versions which describe the families of size close to the extremal value.


For general graphs, in his 1990 personal communication to the authors of \cite{EFF}, Alon observed that if $F$ is a star forest, then there exist $F$-intersecting graph families on $[n]$ with $(1/2-o_n(1))2^{\binom{n}{2}}$ members (see e.g. Construction \ref{cons:starforest} with $k=1$).
He further conjectured that there is an $\varepsilon>0$ such that for every $F$ which is not a star forest, an $F$-intersecting graph family has at most $(1/2-\varepsilon)2^{\binom{n}{2}}$ members.
Alon pointed out that this holds for all non-bipartite graphs $F$ and it suffices to prove the conjecture for the 3-edge path $P_4$ -- indeed, if a graph-family $\F$ is $G$-intersecting and $G$ is bipartite but not a star forest, then $G$ contains $P_4$ as a subgraph, and thus $\F$ is $P_4$-intersecting.
See Alon--Spencer~\cite[p. 271]{AS}, Ellis--Filmus--Friedgut~\cite[§6.3]{EFF}, and the survey of Ellis~\cite[Conjecture 4.7]{EllisSurvey}.

In this note we prove the conjecture.

\begin{theorem}\label{thm:main}
There exists a constant $\varepsilon>0$ such that for every $n$ and every $P_4$-intersecting graph family $\F$,
\[
             |\F|\le \left(\frac12-\varepsilon\right)2^{\binom n2}.
\]
\end{theorem}

Very recently, Keevash~\cite{Keevash} proved that there exists an absolute constant $c>0$ such that every 3AP-intersecting family $\mathcal A\subseteq 2^{[n]}$ satisfies $|\mathcal A|\le (1/2-c)2^n$, giving the first non-trivial progress towards a conjecture of Simonovits and S\'os that the optimal bound is $2^{n-3}$.
His argument is based on the approach of Gillott together with the sumset idea used in the present paper.

\subsection{A few extensions}

We obtain a few extensions to the multi-color setting.

Given a graph $F$ and $k\in \mathbb N$, we call a family $\F$ $(F,k)$-\emph{intersecting} if every $G, H\in \F$ satisfies that every edge-coloring $\chi:E(G\cap H)\rightarrow[k]$ contains a monochromatic copy of $F$.
We prove the following exponential decay for the multi-color version of Theorem~\ref{thm:main}.

\begin{theorem}\label{thm:rP4}
There exists a constant $\lambda\in(1,3/2)$ such that for every $n$ and $k\ge 3$ and every $(P_4,k)$-intersecting graph family $\mathcal F$, $\mu(\mathcal F)\le\frac12\lambda^{-(k-3)/4}.$ 
\end{theorem}

\noindent
\textbf{Remark}. 
We note that the result above is best possible in the sense that there is no such decay when $F$ is a star forest.
Indeed, we show that the asymptotic density sticks to $1/2$ for star forests in the multi-color setting.

\begin{construction}\label{cons:starforest}
Let $H$ be a fixed star forest with $s$ nontrivial components, and let $d$ be the maximum number of leaves in a component of $H$. 
Since $H$ is a subgraph of the disjoint union of $s$ copies of $K_{1,d}$, it suffices to consider star forest $H$ as the union of $s$ copies of $K_{1,d}$ and any fixed $k\in \mathbb N$.
Let $C\ge\max \{ks/2,kd/2\}$ and choose integers $q=q(n)$ and $D=D(n)$ such that $q(D+1)\le n$.  
Take $q$ pairwise vertex-disjoint stars $S_1,\ldots, S_q$ each with $D$ edges, and let $\mathcal F=\{G\subseteq K_n:|\{i\in[q]:|E(G)\cap E(S_i)|\ge D/2+C\}|\ge q/2+C\}$. 
Note that $\mathcal F$ is $(H,k)$-intersecting.
Indeed, for every $G,G'\in \mathcal F$, there are at least $2C$ stars $S_i$ such that $|E(G\cap G')\cap E(S_i)|\ge 2C$.
In every $k$-edge-coloring, each of these stars contains at least $2C/k\ge d$ edges of the same color.
Since there are at least $2C$ stars and only $k$ colors, some color is assigned to at least $2C/k\ge s$ stars.
Thus, these $s$ stars contain a monochromatic copy of $H$.
Finally note that $|\F|\ge (1/2-o_n(1))2^{\binom{n}{2}}$ holds for sufficiently large $q,D$ and $q=o(D)$.
Indeed, recall that for binomial random variable $X\sim \operatorname{Bin}(m,p)$ with $p\in (1/4,3/4)$ and $\ell \in \mathbb N$, we have $\mathbb P[X \ge mp+\ell]\ge 1/2 - O (\ell/\sqrt{m})$.
Then for a uniformly random graph $G$, let $Y$ be the number of stars $S_i$ satisfying $|E(G)\cap E(S_i)|\ge D/2+C$. 
Thus, $Y\sim\operatorname{Bin}(q,p_D)$  where $p_D=1/2 - O(C /\sqrt{D})$.
Hence, $\ell:=\lceil q/2+C-qp_D\rceil=O(C+Cq/\sqrt{D})$ and $\ell/\sqrt q = O(C/\sqrt{q}+C\sqrt{q/D})=o_n(1)$.
Applying the preceding binomial estimate to $Y$
we obtain that $\mu(\F)=\mathbb P[Y\ge q/2+C]\ge\mathbb P[Y\ge qp_D+\ell]\ge1/2-O(\ell/\sqrt{q})=1/2-o_n(1)$. 
\end{construction}

For our next result, recall that the \textit{arboricity} of a graph $H$, denoted by $\arb(H)$, is the minimum number of forests to cover the edge set of $H$.
For subsets $U,V$ of an abelian group, the sum of $U$ and $V$ is
$U+V=\{u+v:u\in U,\ v\in V\}$ and the sum of more than two sets is defined analogously. 
For an integer $j\ge 1$, write $j\cC$ for the $j$-fold sumset of $\cC$, and set $0\cC=\{0\}$. 
The key observation we shall use here is that if a graph $H\notin j\C$, then $\arb(H) > j$.
A family $\cF\subseteq 2^{E(K_n)}$ is called $(\textrm{Arb}, k)$-intersecting if for every $G, H\in \F$, we have $\arb(G\cap H) \ge k$.

We show a similar exponential decay for $(\textrm{Arb}, k)$-intersecting graph families.

\begin{theorem}\label{thm:arboricity}
There exists a constant $\lambda\in(1,3/2)$ such that for every $n,k\in\mathbb{N}$ and every $(\textrm{Arb}, k)$-intersecting graph family $\mathcal F$, $\mu(\mathcal F)\le\frac12\lambda^{-\left\lfloor (k-1)/2\right\rfloor}.$ 
\end{theorem}

Write $m(H)=\sup_{n\ge|V(H)|}
  \max\left\{\mu(\cF):
  \cF\subseteq 2^{E(K_n)}
  \text{ is $H$-intersecting}\right\}$.
Sticking to one graph $F$ gives the following corollary.

\begin{corollary}\label{cor:arb}
There is an absolute constant $\lambda\in(1,3/2)$ such that for every graph $F$ with at least one edge, $m(F)\le\frac12\lambda^{-\left\lfloor (\arb(F)-1)/2\right\rfloor}$.
\end{corollary}

In particular, for the special case $F=K_{t,t}$, as $\arb(K_{t,t}) = \Theta(t)$, it gives $m(K_{t,t})\le \lambda^{-\Theta(t)}$, which addresses a question of Gillott \cite[Question 4.2]{Gillott} that asks for a stronger conclusion whether $m(K_{t,t})\le \lambda^{-\Theta(t^2)}$.

\subsection{Proof ideas}
Gillott~\cite{Gillott} recently proved a constant gap for cycle-intersecting families.
Our proof builds on his approach. 

For a family $\F$ on a set $X$, let its \emph{density} $\mu(\F)$ be $|\F|/2^{|X|} $.
It will be convenient to work with the Boolean setting.
Let $\Gamma=\mathbb F_2^{E(K_n)}$, $M=2^{\binom{n}{2}}$, and
$\mu(A)=|A|/M$ be the uniform measure. 
We identify every graph with the indicator vector of its edge set in $\mathbb F_2^{E(K_n)}$,
and denote the complete graph by $\one$.
Let $\B$ be the family of star forests, and let $\C$ be the forests, including
the empty graph in both cases.

Suppose $\F$ is an $F$-intersecting family and let $\mathcal Q$ be the $F$-free graphs.
Note crucially that
\[
(\F+\mathcal Q) \cap (\F+\one) = \emptyset.
\]
Indeed, an equality $h+q=g+\one$, with $h,g\in\F$ and $q\in\mathcal Q$,
would give $g=h+\one+q$. Every edge in $h\cap g$ then belongs to $q$,
contradicting the definition of $\F$ (in set language, if graphs $H, G\in \F$ and $Q\in \mathcal Q$ satisfy $H\triangle Q = G^c$ where $\triangle$ is the symmetric difference, then $H\cap G = H\setminus G^c = H\setminus (H\triangle Q) \subseteq Q$ is $F$-free).
Therefore, writing $\alpha=\mu(\F)$ and using $\mathbf 0\in \mathcal Q$, we obtain
\begin{equation}
\label{eq:FC}
\frac{|\F+\mathcal Q|}{|\F|} \le \frac{1-\alpha}{\alpha}.
\end{equation}


In his recent work \cite{Gillott}, Gillott showed that (cf. Lemmas~\ref{lem:noise} and~\ref{lem:cycles}), for a given family $\F$ on $[n]$ of density $1/2-\eta$, there exists a family $\mathcal G$ of graphs on $[n]$ of density $1/2+\eta$ such that for every $G\in \mathcal G$, there exists a forest (i.e., cycle-free) $Q$ such that $G\triangle Q\in \F$.
This shows $\mathcal G\subseteq \F+\C$ and thus $\mu(\F+\C)\ge 1/2+\eta$.
The key to the result above is to use a sparse ``random'' graph to show that for more than half of the points $g$ in the cube, such forest $q$ exists.
Here the sparse random graph $G_P$ is obtained by carefully choosing a probability $P_e$ for every $e\in E(K_n)$ along a sequential exposure of the coordinates of a uniformly
random point in the cube.
We collect Gillott's tools and obtain Proposition~\ref{prop:forest}, which says that every family $\F$ with density at least $1/2-\eta$ satisfies that $\mu(\F+\C) \ge 1/2 + \eta$.

For our problem, we can use the star forests $\B$ as boosters\footnote{Note that $\B$ is the family of $P_4$-free and cycle-free graphs, which indeed gives us a slightly stronger result: it gives the conclusion of Theorem \ref{thm:main} for families of graphs such that the intersection of every two members contains a copy of $P_4$ or a cycle.}, but not the forests $\C$.
However, they are closely related via the additive relation $\C\subseteq \B + \B$.
We establish Lemma~\ref{lem:upgrade} which passes the boost from near-half families to small families, which together with Proposition~\ref{prop:forest} immediately implies that forests are good boosters, that is, there exists an absolute $\eta>0$ such that for every family $\F$ with $\mu(\F)\le 1/2$, we have $\mu(\F+\C)\ge (1+\eta)\mu(\F)$.
To get an upper bound on the boost, a key idea is to use a classical result of Pl\"unnecke on sumsets (Lemma~\ref{lem:petridis}) to \textit{pass it to a nonempty subset $X$ of $\F$}: there exists a nonempty $X\subseteq \F$ such that $|X+\B+\B|\le (\frac{|\F+\B|}{|\F|})^2|X|$.
Putting the bounds together and using \eqref{eq:FC} with $\mathcal Q = \B$ we get
\[
(1+\eta)|X|\le |X+\C|\le |X+\B+\B|\le \left(\frac{1-\alpha}{\alpha}\right)^2|X|,
\]
and rearranging gives $\alpha < 1/2 - o_\eta(1)$.
\subsection{Organization}
In Section~\ref{sec:additive}, we collect the sumset input needed in the proof, including Pl\"unnecke's inequality and a general booster lemma.
In Section~\ref{sec:assembly}, we develop the graph-family input: we relate forests to star
forests and, using Gillott's result, establish the forest-boosting property.
In Section~\ref{sec:mainproof}, we combine these ingredients to prove Theorem~\ref{thm:main}.
Section~\ref{sec:furtherproof} contains the proofs of the multi-color and arboricity extensions,
Theorems~\ref{thm:rP4} and~\ref{thm:arboricity}.
Finally, in Section~\ref{sec:concluding}, we discuss a possible extension of the star-forest dichotomy to hypergraphs.

\section{Sumset input}\label{sec:additive}

We first recall an elementary result of Pl\"unnecke \cite{Plunnecke1970}, and a short proof can be found in \cite{Petridis}.

\begin{lemma}[Pl\"unnecke~\cite{Plunnecke1970}]
\label{lem:petridis}
Let $A,B$ be nonempty finite subsets of an abelian group with $\frac{|A+B|}{|A|}\le r$. 
Then there is a nonempty
$X\subseteq A$ such that $|X+B+B|\le r^2|X|$.
\hfill \qedsymbol
\end{lemma}


The following elementary lemma shows that if adding $S$ can boost the density of every near-half family, then it can boost the density of all smaller sets.

\begin{lemma}[Booster]\label{lem:upgrade}
Let $\Gamma$ be a finite abelian group of order $M\ge2$, and let
$S\subseteq\Gamma$ contain $0$ and generate $\Gamma$.
Suppose $0<\eta<1/4$ and
for every $A\subseteq\Gamma$ with $|A|\ge(1/2-\eta)M$, it holds that $|A+S|\ge(1/2+\eta)M$. 
Then every nonempty $X\subseteq\Gamma$
with $|X|\le M/2$ satisfies
\[
                         |X+S|\ge(1+2\eta)|X|.
\]
\end{lemma}

\begin{proof}
Let $\kappa$ be the minimum of $|X+S|/|X|$ over nonempty sets $X$ of size at
most $M/2$. 
Choose a minimizer $X$ of least cardinality. Suppose for a
contradiction that $\kappa<1+2\eta<2$.

We first claim that
\begin{equation}\label{eq:minsize}
                         |X|>\frac{M}{2\kappa}.
\end{equation}
Indeed, otherwise fix any $t\in S$ and set
$U=X\cup(X+t)$, $I=X\cap(X+t)$.
Since $0,t\in S$, we have $U\subseteq X+S$ and hence
$|U|\le\kappa|X|\le M/2$. 
Also $I\ne\varnothing$, since otherwise $2|X|=|U|\le\kappa|X|$, a contradiction with $\kappa<2$.

Note that $U+S=(X+S)\cup (X+t+S)$, and
$I+S\subseteq(X+S)\cap(X+t+S)$. Consequently,
\begin{align*}
 \kappa(|I|+|U|)
 &\le |I+S|+|U+S|\\
 &\le |X+S|+|X+t+S|
   =2\kappa|X|=\kappa(|I|+|U|).
\end{align*}
The first inequality holds since $|I+S|\ge\kappa|I|$ and $|U+S|\ge\kappa|U|$ and the last equality holds by inclusion--exclusion. 
Therefore all inequalities above are equalities, so $I$ is another minimizer. The choice
of $X$ forces $I=X$, yielding $X+t=X$. This holds for every $t\in S$.
Since $S$ generates $\Gamma$ and $X$ is nonempty, it follows that
$X=X+S=\Gamma$, contradicting $|X|\le M/2$. This proves~\eqref{eq:minsize}.

Now~\eqref{eq:minsize} gives
\[
 \frac{|X|}{M}>\frac1{2\kappa}
    >\frac1{2(1+2\eta)}>\frac12-\eta.
\]
By the assumption of the lemma, $|X+S|\ge(1/2+\eta)M$. 
On the other hand, $|X+S|=\kappa|X|\le\kappa M/2<(1/2+\eta)M$, a contradiction.
\end{proof}

\section{Graph-family input}\label{sec:assembly}


In this section we include the following tools.

\begin{lemma}\label{lem:forests-two}
We have $\C\subseteq\B+\B$.
\end{lemma}
\begin{proof}
Let $F\in \C$.
Root each nontrivial component of $F$ arbitrarily. 
Let $\operatorname{depth}(v)$ denote the distance from a vertex $v$ to the root of its component.
For each edge $uv\in E(F)$, where $u$ is the endpoint closer to the root,
assign $uv$ to $F_0$ or $F_1$ according to the parity of $\operatorname{depth}(u)$. 
It is easy to see that both $F_0$ and $F_1$ are star forests, and hence $P_4$-free. 
Since $F_0$ and $F_1$ form a partition of $F$, we have
$F=F_0+F_1\in\B+\B$.
Hence $\C\subseteq \B+\B$.
\end{proof}

For $P\in[0,1]^m$, let $Y_P$ have independent Bernoulli coordinates of parameters $P_1,\ldots,P_m$.
The following result says that given a family $\F$ with $\mu(\F)\ge 1/2-\eta$ and a random point $x$, one can construct adaptively a small noise vector $Y_P$ such that $x+Y_P\in \F$ with a non-trivial probability.

\begin{lemma}\cite[Theorem 3.1]{Gillott}
\label{lem:noise}
There exist $\eta>0$ and $0<L<1/2$ such that the following holds.
Let $\F\subseteq\{0,1\}^m$ have density $\alpha\ge1/2-\eta$.
For a proportion at least $1/2+\eta$ of points $x$, there exists
$P=P(x)\in[0,1]^m$ with
\[
           \sum_i P_i^2\le L,
           \qquad \Prob(x+Y_P\in \F)> \tau:=\frac{(2L)^{3/2}}{1-\sqrt{2L}}.  \tag*{\qedsymbol}
\]
\end{lemma}

For the rest of this note, we take the constants $\eta$ and $L$ and use $\tau=\tau(L)$ from Lemma~\ref{lem:noise}. 
By decreasing $\eta$ if necessary, we may assume that $\eta<1/4$.
The following result is proved and used by Gillott to show that every cycle-intersecting family of graphs has density at most $1/2-o(1)$, see the proof of~\cite[Theorem 3.2]{Gillott}.
We include a short proof here.

\begin{lemma}\cite{Gillott}\label{lem:cycles}
Given $P=(P_e)\in [0,1]^{\binom{n}{2}}$, let $G_P$ be formed by including every edge $e$ independently with probability $P_e$.
If $\sum_{e\in E(K_n)} P_e^2\le L$, then $\mathbb{P}(G_P\text{ contains a cycle})\le \tau$.
\end{lemma}

\begin{proof}
 Let $M=(M_{uv})_{u,v\in[n]}$ be the symmetric matrix with $M_{uv}=P_{uv}$ for $u\ne v$ and 0 otherwise. 
 Then
 $\tr(M^2)=\sum_{u,v\in[n]}M_{uv}^2=2\sum_e P_e^2\le2L$.
 For each $\ell\ge3$,
 let $N_{\ell}$ be the number of $\ell$-cycles in $G_P$.
It is easy to see that $\mathbb E N_\ell=
\sum_{C\in\mathcal C_\ell}
\prod_{e\in E(C)}P_e$ where $\mathcal C_\ell$ is the family of $\ell$-cycles in $K_n$.
Since every $\ell$-cycle contributes at least one to the trace of $M^\ell$ and the other terms are nonnegative,
we have $\mathbb E N_\ell
\le \operatorname{tr}(M^\ell)$.
Let $\lambda_1,\ldots,\lambda_n$ be the real eigenvalues of $M$, and let
$S=\sum_j\lambda_j^2=\tr(M^2)$. 
Since $|\lambda_j|\le\sqrt S$, we have
\[
 0\le\tr(M^\ell)\le\sum_j|\lambda_j|^\ell
 \le S^{(\ell-2)/2}\sum_j\lambda_j^2=S^{\ell/2}\le(2L)^{\ell/2}.
\] 
A union bound over simple cycles, followed by a geometric sum, gives
\[
\mathbb P(G_P\text{ contains a cycle})\le \sum_{\ell=3}^{n}\mathbb E N_\ell\le \sum_{\ell=3}^{n}\operatorname{tr}(M^\ell)\le \sum_{\ell=3}^{\infty}(2L)^{\ell/2}=
\frac{(2L)^{3/2}}{1-\sqrt{2L}}
=\tau.\qedhere
\]
\end{proof}

We combine the previous two results and obtain the following one (again essentially proved by Gillott~\cite{Gillott}).

\begin{proposition}[Forest boost]\label{prop:forest}
Every graph family $\F$ with
$\mu(\F)\ge1/2-\eta$ satisfies $\mu(\F+\C)\ge1/2+\eta$.
\end{proposition}

\begin{proof}
By Lemma~\ref{lem:noise} with $m=\binom n2$, we obtain that for at least $(1/2+\eta)2^m$ points $x\in\{0,1\}^m$, there exists $P=(P_e)_{e\in E(K_n)}\in[0,1]^m$ satisfying $\sum_{e\in E(K_n)}P_e^2\le L$ and $\mathbb{P}\bigl(x+Y_P\in \F\bigr)>\tau$.
We call such points \emph{good}.
Let $G_P$ be the graph whose edge-indicator vector is the random vector $Y_P$.
Then, by Lemma \ref{lem:cycles}, we obtain that $\mathbb{P}\bigl(G_P\notin \C\bigr)\le\tau$.
Therefore, for every good point $x$, we have $\Prob((x+Y_P\in \F)\wedge (G_P\in \C))>0$. 
Hence there exists a forest $q\in\C$ such that $x+q\in \F$.
Since $q+q=0$, it follows that $x=(x+q)+q\in \F+\C$.
We are done as there are at least $(1/2+\eta)2^m$ good points.
\end{proof}

Since $\C$ contains the empty graph and each one-edge, it generates $\Gamma$. Thus Proposition \ref{prop:forest} and Lemma \ref{lem:upgrade} immediately give that forests are good boosters.
\begin{corollary}
[Forest boosters]
\label{cor:fb}
Let $\F\subseteq \{0,1\}^m$ be a family with $\mu(\F)\le 1/2$.
Then $\mu(\F+\C)\ge (1+2\eta)\mu(\F)$. \qed
\end{corollary}

\section{Proof of Theorem~\ref{thm:main}}
\label{sec:mainproof}
Now we are ready to prove our main result.

\begin{proof}[Proof of Theorem~\ref{thm:main}]
Let $\eta<1/4$ be given in Lemma~\ref{lem:noise} and $\varepsilon = \eta/9$.
For $n<4$ the family is empty, so assume $n\ge4$ and $\F\ne\varnothing$.
Put $\alpha=\mu(\F)$. Recall that
\begin{equation*}\label{eq:forbidden}
                    (\F+\B)\cap(\F+\one)=\varnothing.
\end{equation*}
Since $0\in\B$, this gives both $\alpha\le1/2$ and
\begin{equation*}\label{eq:smallgrowth}
                  \frac{|\F+\B|}{|\F|}\le\frac{1-\alpha}{\alpha}.
\end{equation*}

By Lemma~\ref{lem:petridis} applied with $A=\F$ and $B=\B$, there is a nonempty $X\subseteq\F$ with $|X+\B+\B|
       \le\left(\frac{1-\alpha}{\alpha}\right)^2|X|.$
Since $|X|\le M/2$, by Corollary \ref{cor:fb}, we have
\[
 (1+2\eta)|X|\le |X+\C|
       \le |X+\B+\B|
       \le\left(\frac{1-\alpha}{\alpha}\right)^2|X|,
\]
where the second inequality uses Lemma~\ref{lem:forests-two}.
It follows that
\begin{equation*}
 \alpha\le\frac1{1+\sqrt{1+2\eta}}
       =\frac12-\frac{\eta}{(1+\sqrt{1+2\eta})^2} \le \frac12 - \frac{\eta}{9} = \frac{1}{2} -\varepsilon. \qedhere
\end{equation*}
\end{proof}


\section{Further proofs}
\label{sec:furtherproof}
In this section we present the proofs of Theorem \ref{thm:rP4} and Theorem \ref{thm:arboricity}.



\begin{proof}[Proof of Theorem \ref{thm:rP4}]
We may assume that $\cF$ is nonempty.
Set $\Gamma=\FF_2^{E(K_n)}$, $r=\lfloor\frac{k}{2}\rfloor$ and $r'=\lfloor\frac{r}{2}\rfloor$.
We claim that $(\cF+r\B)\cap(\cF+\one+r\B)=\emptyset$.
Otherwise there exist $F,F'\in\cF$ and $Q,Q'\in r\B$ such that $ F+Q=F'+\one+Q'$. 
Every edge belonging to both $F$ and $F'$ must then belong to exactly one of $Q$ and $Q'$.
Hence, $F\cap F'\subseteq Q\triangle Q'$.
The graph $Q\triangle Q'$ is contained in the union of $2r \le k$ star forests, 
so there is a $2r$-coloring of $F\cap F'$ such that every color class is a star forest, contradicting that $\F$ is $(P_4,k)$-intersecting.
Thus, the claim holds and in particular $\mu(\cF+r\B)=\mu(\cF+\one+r\B)\le 1/2$.

Since $0\in\C$, the families $\cF+j\C$, for $0\le j\le r'$, are nested. 
As $\F+r'\C \subseteq\F+2r'\B\subseteq \F+r\B$ by Lemma~\ref{lem:forests-two}, each of them has density at most $1/2$.
Applying Corollary~\ref{cor:fb} successively gives $\lambda^{r'}\mu(\cF)\le \mu(\cF+r'\cC)\le 1/2$ where $\lambda = 1+2\eta$.
Thus, $\mu(\cF)\le\frac12\lambda^{-(k-3)/4}$.
\end{proof}

Next we prove Theorem \ref{thm:arboricity} via a similar idea.

\begin{proof}[Proof of Theorem \ref{thm:arboricity}]
We may assume that $\cF$ is nonempty.
Set $\Gamma=\FF_2^{E(K_n)}$ and $r=\left\lfloor\frac{k-1}{2}\right\rfloor$.
We claim that $(\cF+r\cC)\cap(\cF+\one+r\cC)=\emptyset$.
Otherwise there exist $F,F'\in\cF$ and $Q,Q'\in r\cC$ such that $ F+Q=F'+\one+Q'$. 
Every edge belonging to both $F$ and $F'$ must then belong to exactly one of $Q$ and $Q'$.
Hence, $F\cap F'\subseteq Q\triangle Q'$.
The graph $Q\triangle Q'$ is contained in the union of $2r$ forests, so its arboricity is at most $2r<k$. 
We obtain a contradiction as $k\le \arb(F\cap F')\le \arb(Q\triangle Q')\le 2r<k$.
Thus, $\mu(\cF+r\cC)=\mu(\cF+\one+r\cC)\le 1/2$.

Since $0\in\cC$, the families $\cF+j\cC$, for $0\le j\le r$, are nested. Each thus has density at most $1/2$.
Applying Corollary~\ref{cor:fb} successively gives $\lambda^r\mu(\cF)\le \mu(\cF+r\cC)\le 1/2$ where $\lambda = 1+2\eta$.
Thus, $\mu(\cF)\le\frac12\lambda^{-\left\lfloor (k-1)/2\right\rfloor}$.
\end{proof}


\section{Concluding remarks}
\label{sec:concluding}
In this paper we verified Alon's conjecture and established a dichotomy for $m(F)$, the maximum density of an $F$-intersecting family of graphs.
Let us finish by discussing possible extensions of this dichotomy to uniform hypergraphs.
For $r\ge2$, an $r$-graph $H$ is a \emph{nested sunflower} with core $T$ if, for every two distinct edges $e,f\in E(H)$, $e\cap f\subseteq T$ where $0\le|T|\le r-1$.
A \emph{nested sunflower forest} is a vertex-disjoint union of nested sunflowers.
The cores of different components are allowed to have different sizes.

By an iterated version of the majority construction in Construction \ref{cons:starforest}, we obtain the following proposition.
The main additional complexity is that, unlike a star, a nested sunflower is determined not only by the number of its edges but also by their intersection pattern inside the core. 
We therefore need an iterated majority construction which preserves the relevant intersection types while keeping the resulting family with asymptotic density $1/2$.
\begin{proposition}
\label{prop:nested-sunflower-no-decay}
Let $r\ge2$ and let $H$ be a fixed $r$-uniform nested sunflower forest with at least one edge and no isolated vertices. 
For every sufficiently large $n$, there exists an $H$-intersecting family $\mathcal F$ of $r$-graphs on $[n]$ such that $\mu(\mathcal F)=1/2-o_n(1)$.
\end{proposition}

We conjecture that nested sunflower forests are the only fixed uniform hypergraphs for which the maximum density of an intersecting family can approach $1/2$.

\begin{conjecture}
\label{conj:nested-sunflower-gap}
Let $r\ge2$ and let $H$ be a fixed $r$-graph with at least one edge and no isolated vertices. 
If $H$ is not a nested sunflower forest, then there exists a constant $\varepsilon_H>0$ such that, for every $n\ge |V(H)|$, every $H$-intersecting family $\mathcal F$ of $r$-graphs on $[n]$ satisfies  $\mu(\mathcal F)\le 1/2-\varepsilon_H$.
\end{conjecture}

For $r=2$, nested sunflower forests are precisely star forests, so the conjecture reduces to the star-forest dichotomy for graphs. 

\subsection*{Note on AI-use}
The proof was developed with the help of ChatGPT Pro model in which the authors guided its progress. 

\end{document}